\documentclass[11pt]{amsart}

\usepackage{amssymb}
\usepackage{bm}
\usepackage[centertags]{amsmath}
\usepackage{amsfonts}
\usepackage{amsthm}
\usepackage{mathrsfs}
\usepackage[usenames]{xcolor}
\usepackage[normalem]{ulem}
\usepackage{comment}

\theoremstyle{definition}
\newtheorem{thm}{Theorem}[section]
\newtheorem{defn}[thm]{Definition}
\newtheorem{lem}[thm]{Lemma}
\newtheorem{prop}[thm]{Proposition}
\newtheorem{cor}[thm]{Corollary}
\newtheorem{rem}[thm]{Remark}

\newtheorem{question}{Question}

\newtheorem*{defn*}{Definition}

\newtheorem*{thm*}{Theorem}

\newtheorem*{cor*}{Corollary}

\newtheorem*{prp*}{Proposition}

\newtheorem{problem}{Problem}

\newtheorem*{problem*}{Problem}

\newcommand{\N}{\mathbb{N}}

\newcommand{\iii}[1]{{\left\vert\kern-0.25ex\left\vert\kern-0.25ex\left\vert #1 
    \right\vert\kern-0.25ex\right\vert\kern-0.25ex\right\vert}}

\newcommand\restr[2]{{
  \left.\kern-\nulldelimiterspace 
  #1 
  \vphantom{\big|} 
  \right|_{#2} 
  }}

\long\def\symbolfootnote[#1]#2{\begingroup%
\def\thefootnote{\fnsymbol{footnote}}\footnote[#1]{#2}\endgroup}

\allowdisplaybreaks

\begin{document}

\title[Continuous factorization]{Continuous factorization of the identity matrix}

\author[Y. Dai]{Yuying Dai}
\address{Department of Mathematics, University of Illinois at Urbana-Champaign, Urbana, IL 61801, U.S.A.}
\email{yuyingd2@illinois.edu}

\author[A. Hore]{Ankush Hore}
\address{Department of Mathematics, University of Illinois at Urbana-Champaign, Urbana, IL 61801, U.S.A.}
\email{ahore2@illinois.edu}

\author[S. Jiao]{Siqi Jiao}
\address{Department of Mathematics, University of Illinois at Urbana-Champaign, Urbana, IL 61801, U.S.A.}
\email{sjiao2@illinois.edu}

\author[T. Lan]{Tianxu Lan}
\address{Department of Mathematics, University of Illinois at Urbana-Champaign, Urbana, IL 61801, U.S.A.}
\curraddr{Columbia University, 116th St. \& Broadway, New York, NY 10027} 
\email{tl2971@columbia.edu}

\author[P. Motakis]{Pavlos Motakis}
\address{Department of Mathematics, University of Illinois at Urbana-Champaign, Urbana, IL 61801, U.S.A.}
\email{pmotakis@illinois.edu}

\thanks{This work was done in the research project ``Continuous factorization of the identity matrix'' at the Illinois Geometry Lab in Spring 2019. The first, third, and fourth authors participated as undergraduate scholars, the second author served as graduate student team leaders, and the fifth author as faculty mentor. The project was supported by the National Science Foundation under Grant Number DMS-1449269. The fifth author was  supported by the National Science Foundation under Grant Number DMS-1912897.}

\begin{abstract}
We investigate conditions under which the identity matrix $I_n$  can be continuously factorized through a continuous $N\times N$ matrix function $A$ with domain in $\mathbb{R}$. We study the relationship of the dimension $N$, the diagonal entries of $A$, and the norm of $A$ to the dimension $n$ and the norms of the matrices that witness the factorization of $I_n$ through $A$.
\end{abstract}

\maketitle

\section{Introduction}
The problem from which this paper draws motivation concerns the relation between the magnitude of the diagonal entries $a_{ii}$ of an $N\times N$ matrix $A$, the norm of $A$, and the dimension $n$ of a vector space that $A$ preserves in a satisfying manner, as precisely described below.

\begin{problem}
\label{problem1}
Given $N\in\mathbb{N}$ and $\delta>0$ find the largest $n\in\mathbb{N}$ with the following property: for every $N\times N$ matrix $A = (a_{ij})$ with $\|A\|\leq 1$ the diagonal entries of which satisfy $|a_{ii}|\geq \delta$ for $1\leq i\leq N$, there exist $n\times N$ and $N\times n$ matrices $L$ and $R$ so that $LAR = I_n$ and $\|L\|\|R\| \leq 2/\delta$.
\end{problem}

The  upper bound imposed to the quantity $\|L\|\|R\|$ must necessarily be at least $1/\delta$ (see Remark \ref{why one over delta}). We use elementary combinatorics and linear algebra to study Problem \ref{problem1}. Subsequently, we allow the entries of $A$ to vary continuously and study the corresponding problem in the solution of which it is additionally required that the preserved vector spaces vary continuously as well. In this article we are mainly concerned with the following.

\begin{problem}
\label{problem2}
Given $N\in\mathbb{N}$ and $\delta>0$ find the largest $n\in\mathbb{N}$ with the following property: for every $N\times N$ continuous matrix function $A:\mathbb{R}\to M_N(\mathbb{R})$ with $\|A(t)\|\leq 1$ and $|a_{ii}(t)|\geq \delta$ for $1\leq i\leq N$ and all $t\in \mathbb{R}$, there exist continuous matrix functions $L:\mathbb{R}\to M_{n\times N}(\mathbb{R})$ and $R:\mathbb{R}\to M_{n\times N}(\mathbb{R})$ so that $L(t)A(t)R(t) = I_n$ and $\|L(t)\|\|R(t)\| \leq 2/\delta$ for all $t\in\mathbb{R}$.
\end{problem}
We provide lower bounds for $n$ in Problem \ref{problem1} and Problem \ref{problem2}. In particular, we show that in both cases the order of magnitude of $n$ is at least $\delta^{4/3}N^{1/3}$ (see Theorem \ref{theorem stationary} and Theorem \ref{main theorem}). In the continuous case, this is achieved by using the proof of our estimate for Problem \ref{problem1} point-wise. In this fashion, we obtain an open cover of $\mathbb{R}$ consisting of intervals on each of which there are continuous matrix functions $L$ and $R$ factoring $I_n$ through $A$. In the final step, we use these local solutions as building blocks to construct a continuous solution defined on the entire real line.

Although our approach is entirely Euclidean and finite dimensional, this topic has origins that fit neither description. On a (generally infinite dimensional) Banach space $X$ with a coordinate system $(e_i)_i$ (e.g., a Schauder basis) every bounded linear operator $A:X\to X$ can be identified with an infinite matrix $(a_{ij})$. If this matrix has large diagonal, in the sense that $\inf_i|a_{ii}| >0$, one may ask whether there exist bounded linear operators $L,R:X\to X$ so that $LAR = I_X$. In 1979 A. D. Andrew first showed that the answer is yes if $X = L_p$, $1<p<\infty$ and the coordinate system under consideration is the Haar system (see \cite{Andrew-1979}). Since then, a number of papers have contributed to the study of this general problem in a variety of infinite dimensional Banach spaces $X$ (see, e.g., \cite{Laustsen-Lechner-Mueller-2018}, \cite{Lechner-2018-Israel}, \cite{Lechner-2018-Studia}, \cite{Lechner-2017-arxiv}, \cite{Lechner-2018-arxiv2}, and \cite{Lechner-Mueller-Motakis-Schlumprecht-2018}). The source of the finite dimensional version of this problem can be traced to J. Bourgain and L. Tzafriri. Their paper \cite{Bourgain-Tzafriri-1987}, among other results, provides an estimate for $n$ in Problem \ref{problem1}, which is of the order $\delta^2 N$ (see Remark \ref{Bourain-Tzafriri remark}). Within this context, other finite dimensional non-Euclidean spaces have been studied by R. Lechner in \cite{Lechner-2019-PAMS} and \cite{Lechner-2018-arxiv1}. To the best of our knowledge, the continuous matrix function case has not been considered before.

The paper is divided into two sections. In Section \ref{section constant} we provide necessary estimates for the norm of a matrix as well as estimates for the size of families of columns of a given matrix $A$ with  the property of being almost orthogonal to one another. Subsequently, we proceed to give an estimate of $n$ for Problem \ref{problem1} by defining matrices $L$ and $R$. In Section \ref{section continuous} we explicitly use the definition of $L$ and $R$ of the constant case to find for each $t$ in the domain of the matrix function $A$ $L(t)$ and $R(t)$ as desired. We then extend these solutions continuously on a small interval around $t$. From there on, we synthesize these local solutions by taking appropriate convex combinations of them and we observe that the desired conclusion is satisfied.

In the sequel, for an $N\times N$ matrix $A = (a_{i,j}) = [a_1\cdots a_N]$ we will consider the quantity $\theta = \min_i\|a_i\|$, instead of $\delta = \min_i|a_{i,i}|$. As $\delta \leq \theta$ our results are slightly more general than already advertised. We have included proofs of some well known facts and estimates in an effort to make this paper as self contained as possible. Although all results are stated and proved for matrices with real entries, obvious modifications make them valid for matrices with complex entries as well.

\section{The constant case}
\label{section constant}

We use elementary counting tools and tools from linear algebra to factorize the identity matrix through a square matrix with large diagonal. The section is organized into three subsections. The first one includes simple estimates of the norm of a matrix, the second one presents combinatorial arguments that are used to find collections of columns of a matrix that are almost orthogonal to one another, and in the third one we present the construction of the factors $L$ and $R$ and prove their desired properties.

Let us recall some necessary notions used in this section. We identify $\mathbb{R}^n$ with the collection of $n\times 1$ matrices. Thus when we write $x = (x_1,\ldots,x_n)$ in reality we mean $x = [x_1\cdots x_n]^\top$. For $1\leq i\leq n$ we denote by $e_i$ the vector in $\mathbb{R}^n$ that has $1$ in the $i$'th entry and $0$ in all others. Recall that for a vector $x =  (x_1,\ldots,x_n)$ in $\mathbb{R}^n$ we define its Euclidean norm of $x$ to be the quantity $\|x\| = (\sum_{i=1}^n|x_i|^2)^{1/2}$. For two vectors $x =  (x_1,\ldots,x_n)$, $y =  (y_1,\ldots,y_n)$ in $\mathbb{R}^n$ their inner product is the quantity $\langle x,y\rangle = \sum_{i=1}^nx_iy_i$. The Cauchy-Schwarz inequality states that for such $x$ and $y$ we have $|\langle x,y\rangle| \leq \|x\|\|y\|$ (see, e.g., \cite[Theorem 4.6]{Meckes-2018}). For an $m\times n$ matrix $A = (a_{i,j})$ when we write $A = [a_1\cdots a_n]$ we mean that for  each $1\leq j\leq n$ the entries of the $j$'th column of $A$ form $a_j$, i.e., the vector $(a_{1,j},\ldots,a_{m,j})$ in $\mathbb{R}^m$ (a similar notation can be used for writing $A$ with respect to its rows $\alpha_1^\top,\ldots,\alpha_m^\top$). Then, for $n\in\mathbb{N}$ the $n\times n$ identity matrix $I_n$ is the matrix $[e_1\cdots e_n]$. Recall, if $A$ is an $m\times n$ matrix with columns $a_1,\ldots,a_n$ and $B$ is a $k\times m$ matrix with rows $\beta_1^\top,\ldots,\beta_k^\top$, then the $i,j$'th entry of the product matrix $BA$ is $\langle\beta_i,a_j\rangle$. For an $m\times n $ matrix $A$ we define its norm to be the quantity $\|A\| = \sup\{\|Ax\|: x\in\mathbb{R}^n$, $\|x\|\leq 1\}$. It is easy to see that for $A$ and $x$ of appropriate dimensions we have $\|Ax\|\leq \|A\|\|x\|$. Similarly, by the association property of matrix multiplication (see, e.g., \cite[Theorem 2.10]{Meckes-2018}), for matrices $A$ and $B$ of appropriate dimensions we have $\|AB\| \leq \|A\|\|B\|$. Finally, recall that a function $f:\mathbb{R}\to\mathbb{R}$ is called convex if for every $0\leq\lambda \leq1$ and $s,t\in\mathbb{R}$ we have $f(\lambda s+ (1-\lambda)t) \leq \lambda f(s) + (1-\lambda)\lambda f(t)$. A direct computation can be used to show that the square function $f(t) = t^2$ is a convex function.

\subsection{Upper bounds of matrix norms}
\label{Upper bounds of matrix norms}
The estimates in this subsection are elementary and well known, yet we include the simple proofs for completeness.

\begin{prop}
\label{norm estimate zero}
Let $m$, $n\in\N$ and $A = [a_1\cdots a_n]$ be an $m\times n$ matrix.  Set
\[\Lambda = \max_{1\leq i\leq n}\|a_i\|\;\text{and}\; \lambda = \max_{1\leq i\neq j\leq n}|\langle a_i,a_j\rangle|.\]
Then $\|A\| \leq (\Lambda^2 + (n-1)\lambda)^{1/2}$.
\end{prop}

\begin{proof}
Let $x = (x_1,\ldots,x_m)$ be a vector of norm one. By convexity of the square function we have
\[\Big(\sum_{i=1}^n\frac{1}{n}|x_i|\Big)^2 \leq \frac{1}{n}\sum_{i=1}^n|x_i|^2,\]
or $\sum_{i=1}^n|x_i| \leq n^{1/2}\|x\|$. Then,
\[\begin{split}\|Ax\|^2 &= \langle Ax,Ax\rangle = \sum_{i=1}^mx_i^2\|a_i\|^2 + \sum_{i\neq j}x_ix_j \langle a_i,a_j\rangle\\
&\leq \Lambda^2\|x\|^2 + \lambda\sum_{i\neq j}|x_ix_j| = \Lambda^2 + \lambda(\sum_{i=1}^n|x_i|\sum_{i=j}^n|x_j| - \sum_{i=1}^n|x_i|^2)\\
&\leq \Lambda^2 + \lambda(n - 1).
\end{split}\]
\end{proof}

\begin{cor}
\label{norm estimate one}
Let $n\in\mathbb{N}$ and $A = (a_{i,j})$ be an $m\times n$ matrix. Set $d = \max_{i,j}|a_{i,j}|$. Then $\|A\| \leq d m^{1/2}n^{1/2}$.
\end{cor}

\begin{proof}
Every column of $A$ has norm at most $dm^{1/2}$ and any two different columns have inner product with absolute value at most $md^2$. A direct application of Proposition \ref{norm estimate zero} yields the desired bound.
\end{proof}

\begin{cor}
\label{norm estimate two}
Let $N,n\in\N$ and $A = [a_1\cdots a_n]$ be an $N\times n$ matrix. Set
\[\lambda = \max_{1\leq i\neq j\leq n}|\langle a_i,a_j\rangle|\;\text{and}\;\Delta = \max_{1\leq i \leq n}|\|a_i\|^2 - 1|.\]
Then $\|A^T A - I_n\| \leq n\max\{\lambda,\Delta\}$.
\end{cor}

\begin{proof}
The $i,j$ entry of the matrix $A^T A - I$ is $\langle a_i,a_j\rangle$ if $i\neq j$ and $\|a_i\|^2 - 1$ if $i=j$. The result follows from applying Corollary \ref{norm estimate one}.
\end{proof}

\subsection{Counting arguments}
\label{Counting arguments}
In this section we estimate the maximal number of columns of a norm one matrix that can have large inner product with a fixed column. This estimate is then used to find collections of columns which are almost orthogonal to one another.
\begin{prop}
\label{counting estimate zero}
Let $A = [a_1\cdots a_M]$ be an $N\times M$ matrix and let $\varepsilon >0$. Then for every $1\leq i\leq M$ the set
\[B_{i}^{\varepsilon} = \Big\{1\leq j\leq M: |\langle a_i,a_j\rangle|\geq \varepsilon\Big\}\]
has at most $\|A\|^4/\varepsilon^2$ elements.
\end{prop}

\begin{proof}
If $a_i$ is the zero vector then the conclusion is obvious and we may therefore assume that it is not. Recall that for any matrix $A$ we have $\|A\| = \|A^\top\|$. Indeed, if $x$ is a norm-one vector with $\|A\| = \|Ax\|$ then $\|A\|^2 = \langle Ax,Ax\rangle = \langle x,A^\top Ax\rangle \leq \|x\|\|A^\top Ax\|\leq \|A^\top\|\|A\|\|x\|^2 =  \|A^\top\|\|A\|$ and hence $\|A\| \leq \|A^\top \|$. By symmetry of the argument we also have $\|A^\top\|\leq \|A\|$. We calculate
\[\begin{split}
\|A\|^2 = \|A^T\|^2 \geq \frac{1}{\|a_i\|^2}\|A^Ta_i\|^2 = \frac{1}{\|Ae_i\|^2}\sum_{k=1}^M|\langle a_k,a_i\rangle|^2 \geq \frac{1}{\|A\|^2}\varepsilon^2\#B_i^\varepsilon.
\end{split}\]
\end{proof}

\begin{cor}
\label{combinatorial estimate zero}
Let $n\in\N$ with $n\geq 2$, $0<\varepsilon<1/(n-1)^{1/2}$, and $N\geq n/\varepsilon^2$. Then for any $L\in\N$ and $L\times N$ matrix $A = [a_1\cdots a_N]$ with $\|A\| \leq 1$ there exists $F\subset \{1,\ldots,N\}$ with $\#F = n$ so that for $i\neq j\in F$ we have $|\langle a_i,a_j\rangle| < \varepsilon$.
\end{cor}

\begin{proof}
Set $i_1 = 1$ and inductively pick $i_2,\ldots,i_n$ so that for $2\leq k\leq n$ $i_k\in\{1,\ldots,N\}\setminus(\{i_1,\ldots,i_{k-1}\}\cup(\cup_{m=1}^{k-1} B^\varepsilon_{i_m}))$. This is possible because, by Proposition \ref{counting estimate zero}, in every inductive step $2\leq k\leq n$ the set $\{1,\ldots,N\}\setminus(\{i_1,\ldots,i_{k-1}\}\cup(\cup_{m=1}^{k-1} B^\varepsilon_{i_m}))$ has at least
\[N - \Big(k-1 + \frac{(k-1)}{\varepsilon^2}\Big) \geq \frac{n}{\varepsilon^2} - (n-1)\Big(1+\frac{1}{\varepsilon^2}\Big) = \frac{1}{\varepsilon^2} - (n-1)>0\]
elements.
\end{proof}

The following estimate will be used in Section \ref{section continuous}. We include it here for consistency.
\begin{cor}
\label{combinatorial estimate one}
Let $n\in\N$ with $n\geq 2$, $0<\varepsilon<1/(n-1)^{1/2}$, and $N\geq 5n/\varepsilon^2$. Let $A = [a_1\cdots a_N]$ be an $N\times N$ matrix with $\|A\|\leq 1$. Then for every $F_1$, $F_2\subset\{1,\ldots,N\}$ with $\#F_1 = \#F_2 = n$ there exists $F_3\subset\{1,\ldots,N\}$ with $\#F_3 = n$ so that the following hold:
\begin{itemize}
\item[(i)] $F_3$ is disjoint from $F_1\cup F_2$,
\item[(ii)] for any $i\neq j\in F_3$ we have $|\langle a_i,a_j\rangle|<\varepsilon$, and
\item[(iii)] for any $i\in F_3$, $j\in F_1\cup F_2$ we have $|\langle a_i,a_j\rangle|<\varepsilon$.
\end{itemize}
\end{cor}

\begin{proof}
Define $G = \{1,\ldots,N\}\setminus((F_1\cup F_2)\cup(\cup_{i\in F_1\cup F_2}B_i^\varepsilon))$. Then $\#G \geq 5n/\varepsilon^2 - 2n - 2n/\varepsilon^2 = 3n/\varepsilon^2 - 2n\geq n/\varepsilon^2$. We now follow the exact same argument as in the proof of Corollary  \ref{combinatorial estimate zero} to find $F_3\subset G$ with $\#F_3 = n$ so that for all $i\neq j\in F_3$ we have $|\langle a_i,a_j\rangle| <\varepsilon$. The fact that $F_3\subset G$ also yields (i) and (iii).
\end{proof}

\subsection{The matrices $L$ and $R$}
We next explicitly define the matrices $L$ and $R$ with the property $LAR = I_n$. For the definition of $L$ and $R$ we use the results from Subsection \ref{Counting arguments}. We then use the estimates provided in Subsection \ref{Upper bounds of matrix norms} to estimate the quantity $\|L\|\|R\|$.

We now introduce the matrices $L_{(A,F)}$, $R_{(A,F)}$ that are defined using $A$ and a subset $F$ of the columns of $A$. This dependence on $F$ will also be important in the next section.

\begin{defn}
\label{left and right}
Let $n\leq N\in\N$, $A = [a_1\cdots a_N]$ be an $N\times N$ matrix, and let $F = \{i_1<\cdots<i_n\}$ be a subset of $\{1,\ldots,N\}$ with $\|a_i\| > 0$ for $i\in F$. For $k=1,\ldots,n$ set $r_{(A,F)}^k = e_{i_k}/\|a_{i_k}\|$, i.e., the $N$-dimensional vector  that has $1/\|a_{i_k}\|$ in the $i_k$'th entry and zero everywhere else. Define the $N\times n$ and $n\times N$ matrices
\[R_{(A,F)} = \Big[r_{(A,F)}^1\cdots r_{(A,F)}^n\Big]\quad\text{and}\quad L_{(A,F)} = (AR_{(A,F)})^T.\]
\end{defn}

\begin{rem}
\label{AR form}
Observe that for $1\leq k\leq n$ we have $Ar_{(A,F)}^k = a_{i_{k}}/\|a_{i_{k}}\|$ and thus
\[AR_{(A,F)} = \Big[\frac{a_{i_1}}{\|a_{i_1}\|}\cdots \frac{a_{i_n}}{\|a_{i_n}\|}\Big].\]
\end{rem}

Here, we give estimates for the norms of the matrices $L_{(A,F)}$, $R_{(A,F)}$, and $L_{(A,F)}AR_{(A,F)} - I_n$.

\begin{prop}
\label{F matrices estimates}
Let $n\leq N\in\N$, let $A$ be an $N\times N$ matrix, and let $F = \{i_1<\cdots<i_n\}$ be a subset of $\{1,\ldots,N\}$ with $\|a_i\| > 0$ for $i\in F$. Set
\[\theta = \min_{i\in F}\|a_i\|\;\text{and}\;\varepsilon = \max_{i\neq j\in F}|\langle a_i,a_j\rangle|.\]
Then we have
\[
\begin{split}
&\|R_{(A,F)}\| \leq \theta^{-1},\; \|L_{(A,F)}\| \leq 1 + \frac{(n-1)^{1/2}\varepsilon^{1/2}}{\theta},\;\text{and}\\
&\|L_{(A,F)}AR_{(A,F)} - I_n\|\leq \frac{n\varepsilon}{\theta^2}.
\end{split}
\]
\end{prop}

\begin{proof}
The first two estimates follow from Proposition \ref{norm estimate zero} whereas the third is a consequence of Corollary \ref{norm estimate two}. For the first one observe that the columns of $R_{(A,F)}$ all have norm at most $1/\theta$ and they are all orthogonal to one another. For the second one, if we denote $b_k = a_{i_k}/\|a_{i_k}\|$ for $1\leq k\leq n$ then by Remark \ref{AR form}
\[AR_{(A,F)} = [b_1\cdots b_n].\]
That is, all columns of $AR_{(A,F)}$ have norm one and for $1\leq k\neq m\leq n$ we have $|\langle b_k,b_m\rangle| \leq \varepsilon/\theta^2$. Recall that for $x\geq 0$ we have $(1+x)^{1/2}\leq 1+ x^{1/2}$. Thus,
\[\|L_{(A,F)}\| = \|L_{(A,F)}^T\| \leq (1 + (n-1)\varepsilon/\theta^2)^{1/2} \leq 1 + (n-1)^{1/2}\varepsilon^{1/2}/\theta.\]
The final estimate follows from Corollary \ref{norm estimate two} directly applied to the matrix $AR_{(A,F)}= [b_1\cdots b_n]$.
\end{proof}

The following is the main result of this section.

\begin{thm}
\label{theorem stationary}
Let $N\in\N$ and let $A =[a_1\cdots a_N]$ be an $N\times N$ matrix with $\|A\| \leq 1$. If $\theta = \min_{1\leq i\leq N}\|a_i\| > 0$ then for every  $1\leq n\leq \frac{1}{5}\theta^{4/3}N^{1/3}$ there exist $n\times N$ and $N\times n$ matrices $L$ and $R$ respectively so that $LAR = I_n$ and $\|L\|\|R\| \leq 2/\theta$.
\end{thm}

\begin{proof}
If $n = 1$ the result easily follows by picking any column $a_i$ and defining $R = e_i/\|a_i\|$, $L = a_i^T/\|a_i\|$. We will therefore assume that $2\leq n\leq \frac{1}{5}\theta^{4/3}N^{1/3}$. Define $\varepsilon = \theta^2/(9(n-1))$. This choice of $\varepsilon$ assures that
\begin{equation}
\label{error1}
\frac{(n-1)^{1/2}\varepsilon^{1/2}}{\theta} = \frac{1}{3}\quad\text{and}\quad \frac{n\varepsilon}{\theta^2}\leq \frac{1}{4}.
\end{equation}
The above two estimates will be used as assumptions to apply Proposition \ref{F matrices estimates}, however, we will first use Corollary \ref{combinatorial estimate zero}. For that purpose, the choice of $\varepsilon$ assures that
\[\frac{n}{\varepsilon^2} = \frac{81n(n-1)^2}{\theta^4} \leq \frac{81}{\theta^4}n^3 \leq \frac{81}{\theta^4}\frac{\theta^4N}{125}\leq N,\]
i.e., $N\geq n/\varepsilon^2$. It is also easily checked that $\varepsilon < 1/(n-1)^{1/2}$ (because $0<\theta\leq 1$). Thus, by Corollary \ref{combinatorial estimate zero}, there exists $F\subset \{1,\ldots,N\}$ with $\#F = n$ so that for $i\neq j\in F$ we have $|\langle a_i,a_j\rangle| <\varepsilon$.

Consider now the matrices $L_{(A,F)}$ and $R_{(A,F)}$ given by Definition \ref{left and right}. By Proposition \ref{F matrices estimates} and \eqref{error1} we deduce
\begin{equation}
\|R_{(A,F)}\| \leq \theta^{-1},\|L_{(A,F))}\|\leq 4/3\;\text{ and }\|L_{(A,F)}AR_{(A,F)} - I_n\|\leq \frac{1}{4}.
\end{equation}
Set $R = R_{(A,F)}$. To define $L$, recall that if $S$ is an $n\times n$ matrix with $\|S - I_n\| = c<1$ then $S^{-1}$ exists and $\|S^{-1}\| \leq 1/(1-c)$. One way to see this is to observe that $S^{-1} = \sum_{k=0}^\infty(I-S)^k$. Therefore, the matrix $(L_{(A,F)}AR_{(A,F)})^{-1}$ is well defined and has norm at most $1/(1-1/4) = 4/3$. Finally, set $L = (L_{(A,F)}AR_{(A,F)})^{-1}L_{(A,F)}$ and observe that $LAR = I_n$, $\|R\|\leq 1/\theta$, and $\|L\| \leq 16/9\leq 2$.
\end{proof}

\begin{rem}
\label{Bourain-Tzafriri remark}
The above theorem may also be stated for an $N\times N$ matrix $A$ without restrictions on $\|A\|$ as follows: if $\theta = \min_{1\leq i\leq N}\|a_i\| > 0$ then for every $1\leq n\leq \frac{1}{5}(\theta/\|A\|)^{4/3}N^{1/3}$ there exist $n\times N$ and $N\times n$ matrices $L$ and $R$ respectively so that $LAR = I_n$ and $\|L\|\|R\| \leq 2\|A\|/\theta$. This estimate can be compared to \cite[Theorem 1.2]{Bourgain-Tzafriri-1987}, which yields a similar result: there exist universal constants $c,C>0$ so that if $N$, $A$, and $\theta$ are is as above then for every $1\leq n \leq c(\theta/\|A\|)^2N$ there exist $n\times N$ and $N\times n$ matrices $L$ and $R$ respectively so that $LAR = I_n$  and $\|L\|\|R\| \leq C\|A\|/\theta$. We observe that the result from \cite{Bourgain-Tzafriri-1987} gives a better relation between the dimension $n$ and $N$ whereas our result gives a better relation between $n$ and the quantity $\theta/\|A\|$.
\end{rem}

\begin{rem}
\label{why one over delta}
In Theorem \ref{theorem stationary} whenever $n\geq 2$ then the quantity $\|L\|\|R\|$ can not be demanded to be below $1/\theta$. To see this fix $0<\theta\leq 1$ and consider the $N\times N$ diagonal matrix $A$ with first diagonal entry $1$ and all other diagonal entries $\theta$. If $n\geq 2$ and we assume that $L$, $R$ are matrices with $LAR = I_n$ then consider the subspace $X$ of $\mathbb{R}^n$ of all vectors orthogonal to $R^\top e_1$. Then $X$ has codimension at most one and in particular it is non-trivial, i.e., we may pick $x\in X$ with $\|x\| = 1$. Then, $Rx = \sum_{i=1}^n\langle e_i,Rx\rangle e_i = \sum_{i=2}^n\langle e_i,Rx\rangle e_i$ and thus we can compute that $ARx = \sum_{i=2}^n\theta \langle e_i,Rx\rangle e_i = \theta Rx$. By assumption, $LAR = I_n$ and so $\|x\| = \|LARx\| = \theta\|LRx\|\leq \theta \|L\|\|R\|\|x\|$. We conclude $\|L\|\|R\| \geq 1/\theta$.
\end{rem}

\section{The continuous case}
\label{section continuous}

In this section we present the main result of our paper. We demonstrate how the estimates from the previous section can be utilized to continuously factor the identity matrix through a continuous matrix function $A = A(t)$ with large diagonal entries. The idea behind the argument is to first obtain continuous factors $L(t)$, $R(t)$ on small intervals that cover the real line and then stitch the different solutions together in a continuous manner.

Let us recall the notion of a matrix function. We denote by $M_{m\times n}(\mathbb{R})$ the set consisting of all $m\times n$ matrices with real entries. We will write $M_N(\mathbb{R})$ instead of $M_{N\times N}(\mathbb{R})$. A matrix function $A$ is a function with some domain $D$ and range in some  $M_{m\times n}(\mathbb{R})$, i.e., it maps every $t\in D$ to some $m\times n$ matrix $A(t) = (a_{i,j}(t))$. Whenever the domain $D$ is equipped with a topology (e.g., when $D$ is a subset of $\mathbb{R}$ with the usual distance) then we say that a matrix function $A$ is continuous whenever all its entries $a_{i,j}$, viewed as scalar functions with domain $D$, are continuous. It is straightforward that for continuous matrix functions $A$, $B$ with appropriate dimensions and common domain $D$ the product $AB$ is a continuous matrix function.

The first Proposition of this section infers that to prove the main result it is enough to find continuous factors $L(t)$, $R(t)$ so that $L(t)A(t)R(t)$ is sufficiently close to the identity matrix for all $t$. We begin with two well known lemmas, which we prove for the sake of completeness.

\begin{lem}
\label{entry-wise continuous}
Let $I$ be an interval of $\mathbb{R}$, $m,n\in\mathbb{N}$ and $A:I\to M_{m\times n}(\mathbb{R})$ be a matrix function. For any $t_0$ in $I$ the matrix function $A$ is continuous at $t_0$ if and only if $\lim_{t\to t_0}\|A(t) - A(t_0)\| = 0$.
\end{lem}

\begin{proof}
Note that for any $m\times n$ matrix $B = (b_{i,j})$ and any $1\leq i_0\leq m$, $1\leq j_0\leq n$ we have $|b_{i_0,j_0}(t)| = |\langle e_{i_0}, Be_{j_0}\rangle| \leq \|B\|$. By Corollary \ref{norm estimate one} we also have $\|B\| \leq m^{1/2}n^{1/2}\max_{i,j}|b_{i,j}|$. For $t\in I$ we apply our observation to be matrix $B = A(t) - A(t_0)$ to obtain that for any $1\leq i_0\leq m$, $1\leq j_0\leq n$ we have
\[|a_{i_0,j_0}(t) - a_{i_0,j_0}(t_0)| \leq \|A(t) - A(t_0)\| \leq m^{1/2}n^{1/2}\max_{i,j}|a_{i,j}(t) - a_{i,j}(t_0)|.\]
The desired conclusion immediately follows.
\end{proof}

\begin{lem}
\label{continuity of inverse}
Let $N\in\N$, $I$ be an interval of $\mathbb{R}$, and $A:I\to M_{N}(\mathbb{R})$ be a continuous matrix function such that $A(t)$ is invertible for all $t\in I$. Then $A^{-1}:I\to M_{N}(\mathbb{R})$ is a continuous matrix function.
\end{lem}

\begin{proof}
We fix $t_0$ in $I$ and estimate $\|A^{-1}(t) - A^{-1}(t_0)\|$ for $t$ close to $t_0$.
Observe that $A^{-1}(t) - A^{-1}(t_0) = A^{-1}(t)(A(t_0) - A(t))A^{-1}(t_0)$. We deduce
\begin{equation}
\label{continuity of inverse eq1}
\|A^{-1}(t) - A^{-1}(t_0)\| \leq \|A^{-1}(t)\|\|A(t_0) - A(t)\|\|A^{-1}(t_0)\|\;\text{and}
\end{equation}
\begin{equation*}
\|A^{-1}(t)\| \leq \| A^{-1}(t_0)\| + \|A^{-1}(t)\|\|A(t_0) - A(t)\|\|A^{-1}(t_0)\|,
\end{equation*}
which, solving for $\|A^{-1}(t)\|$, yields
\begin{equation}
\label{continuity of inverse eq2}
\|A^{-1}(t)\| \leq \frac{\| A^{-1}(t_0)\|}{1-\|A(t_0) - A(t)\|\|A^{-1}(t_0)\|}.
\end{equation}
The quantity on the right hand side of the above inequality is well defined for $t$ sufficiently close to $t_0$. We plug \eqref{continuity of inverse eq2} into \eqref{continuity of inverse eq1} to get rid of the term $\|A^{-1}(t)\|$:
\begin{equation*}
\|A^{-1}(t) - A^{-1}(t_0)\| \leq \frac{\| A^{-1}(t_0)\|^2\|A(t_0) - A(t)\|}{(1-\|A(t_0) - A(t)\|\|A^{-1}(t_0)\|)}.
\end{equation*}
This estimate, in conjunction with Lemma \ref{entry-wise continuous}, yields that the continuity of $A:I\to M_{N}(\mathbb{R})$ at $t_0$ implies the continuity of $A^{-1}:I\to M_{N}(\mathbb{R})$ at $t_0$.
\end{proof}

\begin{prop}
\label{approximation is enough}
Let $n\leq N\in\N$, $I$ be an interval of $\mathbb{R}$, and $A:I\to M_{N}(\mathbb{R})$ be a continuous matrix function. Assume that $0<C<1$, $\Delta\geq 0$, and $L:I\to M_{n\times N}(\mathbb{R})$, $R:I\to M_{N\times n}(\mathbb{R})$ are continuous matrix functions so that for all $t\in I$ we have $\|L(t)A(t)R(t) - I_n\| \leq C$ and $\|L(t)\|\|R(t)\| \leq \Delta$. Then there exist continuous matrix functions $\tilde L:I\to M_{n\times N}(\mathbb{R})$, $\tilde R:I\to M_{N\times n}(\mathbb{R})$ so that for all $t\in I$ we have $\tilde L(t)A(t)\tilde R(t) = I_n$ and $\|\tilde L\|\|\tilde R\| \leq \Delta/(1-C)$.
\end{prop}

\begin{proof}
For each $t\in I$, because we have that $\|L(t)A(t)R(t) - I_n\| \leq C$, the matrix $L(t)A(t)R(t)$ is invertible, and in particular $\|(L(t)A(t)R(t))^{-1}\|\leq 1/(1-C)$. By Lemma \ref{continuity of inverse} the matrix function $(LAR)^{-1}:I\to M_n(\mathbb{R})$ is continuous. We define $\tilde L:I\to M_{n\times N}(\mathbb{R})$ as $\tilde L(t) =  (L(t)A(t)R(t))^{-1}L(t)$ and just set $\tilde R = R$. Both $\tilde L$ and $\tilde R$ are continuous and clearly for all $t\in I$ we have $\tilde L(t)A(t)\tilde R(t) = I_n$. Additionally, for $t\in I$ we have $\|\tilde L(t)\|\|\tilde R\| \leq \|(L(t)A(t)R(t))^{-1}\|\|L\|\|R\| \leq \Delta/(1-C)$.
\end{proof}

Recall the matrices $L_{(A,F)}$ and $R_{(A,F)}$ from Definition \ref{left and right}. In the sequel we will start with two versions of pairs $L_{(A,F_1)}$, $R_{(A,F_1)}$, $L_{(A,F_2)}$ and $R_{(A,F_2)}$, and a scalar $0\leq \lambda\leq 1$. We will combine them into a new pair $L^\lambda_{(A,F_1,F_2)}$ and $R^\lambda_{(A,F_1,F_2)}$.

\begin{defn}
\label{left and right mix}
Let $n\leq N\in\N$, $A = [a_1\cdots a_N]$ be an $N\times N$ matrix, let $F_1 = \{i_1<\cdots<i_n\}$, $F_2 = \{j_1<\cdots<j_n\}$ be disjoint subsets of $\{1,\ldots,N\}$, and let $0\leq \lambda\leq 1$. We assume that $\|a_i\| > 0$ for $i\in F_1\cup F_2$. Define the $N\times n$ and $n\times N$ matrices
\[\begin{split}
R^\lambda_{(A,F_1,F_2)} &= \lambda^{1/2} R_{(A,F_1)} + (1-\lambda)^{1/2}R_{(A,F_2)}\;\text{and}\\
L^\lambda_{(A,F_1,F_2)} &= \lambda^{1/2} L_{(A,F_1)} +(1-\lambda)^{1/2}L_{(A,F_2)}.
\end{split}\]
\end{defn}

\begin{rem}
\label{endpoints are right}
The matrices $R_{(A,F_1,F_2)}^\lambda$, $L_{(A,F_1,F_2)}^\lambda$ lie ``between'' $R_{(A,F_1)}$, $R_{(A,F_2)}$ and $L_{(A,F_1)}$, $L_{(A,F_2)}$ respectively. Clearly, if $\lambda = 1$ then
\[R^1_{(A,F_1,F_2)} = R_{(A,F_1)},\quad L^1_{(A,F_1,F_2)} = L_{(A,F_1)}\]
and if $\lambda = 0$ then
\[R^0_{(A,F_1,F_2)} = R_{(A,F_2)},\quad L^0_{(A,F_1,F_2)} = L_{(A,F_2)}.\]
\end{rem}

\begin{rem}
\label{mix AR form}
Recall that for $k=1,\ldots,n$, $R_{(A,F_1)}e_k = e_{i_k}/\|a_{i_k}\|$ and $R_{(A,F_2)}e_k = e_{j_k}/\|a_{j_k}\|$ which means that  $R_{(A,F_1,F_2)}^\lambda e_k = \lambda^{1/2}e_{i_k}/\|a_{i_k}\|+ (1-\lambda)^{1/2}e_{j_k}/\|a_{j_k}\|$. Therefore
\[AR_{(A,F_1,F_2)}^\lambda = \Big[\Big(\lambda^{\frac{1}{2}}\frac{a_{i_1}}{\|a_{i_1}\|}+(1-\lambda)^{\frac{1}{2}}\frac{a_{j_1}}{\|a_{j_1}\|}\Big)\cdots \Big(\lambda^{\frac{1}{2}}\frac{a_{i_n}}{\|a_{i_n}\|}+(1-\lambda)^{\frac{1}{2}}\frac{a_{j_n}}{\|a_{j_n}\|}\Big)\Big].\]
\end{rem}

\begin{rem}
\label{still continuous}
It will be important to note for the sequel the following: if $n\leq N\in\N$, $I$ is an interval of $\mathbb{R}$, $\lambda:I\to [0,1]$ is a continuous scalar function, $A = [a_1\cdots a_N]:I\to M_{N}(\mathbb{R})$ is a continuous matrix function, and $F_1$, $F_2$ are disjoint subsets of $\{1,\ldots,N\}$ with $\#F_1 = \#F_2 = n$ so that $\|a_i(t)\| > 0$ for all $i\in F_1\cup F_2$ and $t\in I$, then the matrix functions $R_{(F_1,F_2,A(t))}^{\lambda(t)}:I\to M_{N\times n}(\mathbb{R})$, $L_{(F_1,F_2,A(t))}^{\lambda(t)} :I\to M_{n\times N}(\mathbb{R})$ are both continuous.
\end{rem}

The following proposition basically states that if we have appropriately picked $L_{(A,F_1)}$, $R_{(A,F_1)}$, $L_{(A,F_2)}$ and $R_{(A,F_2)}$ then for any scalar $0\leq \lambda\leq 1$ the new pair $L^\lambda_{(A,F_1,F_2)}$, $R^\lambda_{(A,F_1,F_2)}$ satisfies a conclusion similar to that of Proposition \ref{F matrices estimates}.

\begin{prop}
\label{mix matrices estimates}
Let $n\leq N\in\N$, $A = [a_1\cdots a_N]$ be an $N\times N$ matrix, $F_1 = \{i_1<\cdots<i_n\}$, $F_2 = \{j_1<\cdots<j_n\}$ be  disjoint subsets of $\{1,\ldots,N\}$ and let $0\leq \lambda\leq 1$. Set
\[\theta = \min_{i\in F_1\cup F_2}\|a_i\|\;\text{and}\;\varepsilon = \max_{i\neq j\in F_1\cup F_2}|\langle a_i,a_j\rangle|.\]
If $\theta >0$ then we have
\[
\begin{split}
&\|R^\lambda_{(A,F_1,F_2)}\| \leq \theta^{-1},\; \|L^\lambda_{(A,F_1,F_2)}\| \leq 1 + \frac{(2n)^{1/2}\varepsilon^{1/2}}{\theta},\;\text{and}\\
&\|L^\lambda_{(A,F_1,F_2)}AR^\lambda_{(A,F_1,F_2)} - I_n\|\leq \frac{2n\varepsilon}{\theta^2}.
\end{split}
\]
\end{prop}

\begin{proof}
This proof is very similar in spirit to that of Proposition \ref{F matrices estimates}. We examine for $1\leq k\leq n$ column  $k$ of $R_{(A,F_1,F_2)}^\lambda$, i.e. the vector $R^\lambda_{(A,F_1,F_2)}e_k$:
\[\|R^\lambda_{(A,F_1,F_2)}e_k\|^2 = \lambda/\|a_{i_{k}}\|^2 + (1-\lambda)/\|a_{j_{k}}\|^2 \leq 1/\theta^2.\]
It is also easy to see that for $k_1\neq k_2$ the columns of $R_{(A,F_1,F_2)}^\lambda$ are orthogonal. Therefore, by Proposition \ref{norm estimate zero} we have $\|R_{(A,F_1,F_2)}^\lambda\|\leq 1/\theta$.

For the second estimate, we denote, for $1\leq k\leq n$, $b_k = \lambda^{1/2}a_{i_{k}}/\|a_{i_{k}}\| + (1-\lambda)^{1/2}a_{j_{k}}/\|a_{j_{k}}\|$. By Remark \ref{mix AR form} we have that
\[(L^\lambda_{(A,F_1,F_2)})^T = AR^\lambda_{(A,F_1,F_2)} = [b_1\cdots b_k].\]

We calculate, for $1\leq k\leq n$, the norm of column $k$:
\[\begin{split}
\|b_{k}\|^2 &= \Big\langle \frac{\lambda^{1/2}}{\|a_{i_{k}}\|}a_{i_{k}} + \frac{(1-\lambda)^{1/2}}{\|a_{j_{k}}\|}a_{j_{k}}, \frac{\lambda^{1/2}}{\|a_{i_{k}}\|}a_{i_{k}} + \frac{(1-\lambda)^{1/2}}{\|a_{j_{k}}\|}a_{j_{k}}\Big\rangle\\
& = \lambda + (1-\lambda) + 2\lambda^{1/2}(1-\lambda)^{1/2}\langle \frac{a_{i_{k}}}{\|a_{i_{k}}\|}, \frac{a_{j_{k}}}{\|a_{j_{k}}\|}\rangle,
\end{split}\]
That is,
\begin{equation}
\label{close to one}
\Big|\|b_{k}\|^2 - 1\Big| \leq 2\lambda^{1/2}(1-\lambda)^{1/2}\frac{\varepsilon}{\theta^2} \leq \frac{\varepsilon}{\theta^2},\;\text{for}\;1\leq k\leq n,
\end{equation}
where we used $0\leq 2\lambda^{1/2}(1-\lambda)^{1/2} \leq 1$ for $0\leq \lambda\leq 1$. In particular, we have
\begin{equation}
\label{not too big}
\|b_{k}\| \leq \Big(1 + \frac{\varepsilon}{\theta^2}\Big)^{1/2}\;\text{for}\; 1\leq k\leq n.
\end{equation}
 Next, we will show that
 \begin{equation}
 \label{sufficiently orthogonal}
 \text{for $1\leq k_1\neq k_2\leq n$ we have $|\langle b_{k_1},b_{k_2}\rangle| \leq 2\varepsilon/\theta^2$.}
 \end{equation}
 We have
 \[\begin{split}
 |\langle b_{k_1},b_{k_2}\rangle| &\leq \lambda\Big|\Big\langle \frac{a_{i_{k_1}}}{\|a_{i_{k_1}}\|}, \frac{a_{i_{k_2}}}{\|a_{i_{k_2}}\|}\Big\rangle\Big| + (1-\lambda)\Big|\Big\langle \frac{a_{j_{k_1}}}{\|a_{j_{k_1}}\|}, \frac{a_{j_{k_2}}}{\|a_{j_{k_2}}\|}\Big\rangle\Big|\\
 &+ \lambda^{1/2}(1-\lambda)^{1/2}\Big(\Big|\Big\langle\frac{a_{i_{k_1}}}{\|a_{i_{k_1}}\|},\frac{a_{j_{k_2}}}{\|a_{j_{k_2}}\|}\Big\rangle\Big| + \Big|\Big\langle\frac{a_{j_{k_1}}}{\|a_{j_{k_1}}\|},\frac{a_{i_{k_2}}}{\|a_{i_{k_2}}\|}\Big\rangle\Big|\Big)\\
 &\leq \frac{\varepsilon}{\theta^2} + 2\lambda^{1/2}(1-\lambda)^{1/2}\frac{\varepsilon}{\theta^2}\leq 2\frac{\varepsilon}{\theta^2}.
 \end{split}\]

We now apply Proposition \ref{norm estimate zero}, which by \eqref{not too big} and \eqref{sufficiently orthogonal}, gives that
\[\begin{split}
\|L_{(A,F_1,F_2)}^\lambda\| &= \|AR^\lambda_{(A,F_1,F_2)}\| \leq \Big(1+\frac{\varepsilon}{\theta^2} + (n-1)2\frac{\varepsilon}{\theta^2}\Big)^{1/2}\\
&\leq 1 + (2n-1)^{1/2}\frac{\varepsilon^{1/2}}{\theta} \leq 1+ (2n)^{1/2}\frac{\varepsilon^{1/2}}{\theta}.
\end{split}\]

The final estimate follows from Corollary \ref{norm estimate two} directly applied to the matrix $AR_{(A,F_1,F_2)}^\lambda= [b_1\cdots b_k]$ and \eqref{close to one}, \eqref{sufficiently orthogonal}.
\end{proof}

We are finally ready to state and prove the main result of this paper.

\begin{thm}
\label{main theorem}
Let $N\in\N$, let $I$ be an interval of $\mathbb{R}$ and let $A = [a_1\cdots a_N]:I\to M_N(\mathbb{R})$ be a continuous function so that the following hold:
\begin{itemize}
\item[(i)] For $t\in I$ we have $\|A(t)\| \leq 1$ and
\item[(ii)] $\theta = \inf_{t\in I}\min_{1\leq i \leq N}\|a_i(t)\| > 0$.
\end{itemize}
Then for every $1\leq n\leq \frac{1}{12}\theta^{4/3}N^{1/3}$ there exist continuous functions $L:I\to M_{n\times N}(\mathbb{R})$ and $R:I\to M_{N\times n}(\mathbb{R})$ so that for all $t\in I$ we have $L(t)A(t)R(t) = I_n$ and $\|L(t)\|\|R(t)\| \leq 2/\theta$.
\end{thm}

\begin{proof}
By Proposition \ref{approximation is enough} it is sufficient to find continuous $L(t)$, $R(t)$ so that for all $t\in I$ we have $\|L(t)A(t)R(t) - I_n\| \leq 1/4$ and $\|L(t)\|\|R(t)\| \leq 4/(3\theta)$.

The case $n=1$ is treated easily by taking an arbitrary $1\leq i\leq N$ and defining $R(t) = e_i/\|a_i(t)\|$ and $L(t) = a_i(t)/\|a_i(t)\|$, thus we assume that $2\leq n\leq \frac{1}{12}\theta^{4/3}N^{1/3}$. Define $\varepsilon = \theta^2/(18n)$. This choice of $\varepsilon$ is related to the estimates from Proposition \ref{mix matrices estimates} and also Corollaries \ref{combinatorial estimate zero} and \ref{combinatorial estimate one}. Let us note that we have
\begin{eqnarray}
\label{errors good zero}
\frac{(2n)^{1/2}\varepsilon^{1/2}}{\theta} = \frac{1}{3}\quad\text{and}\quad\frac{2n\varepsilon}{\theta^2} \leq \frac{1}{4}\;\text{ and also}\\
5\frac{n}{\varepsilon^2} = 5\frac{18^2n^3}{\theta^4} \leq 5\frac{18^2}{\theta^4}\frac{\theta^4 N}{12^3} \leq N.
\end{eqnarray}
Let us assume henceforth that $I = [0,\infty)$. The case $I =\mathbb{R}$ is treated by performing the same argument on both sides of $0$. Other cases are treated similarly. Otherwise they can be deduced from the previous two cases by using, e.g., that any open interval is homeomorphic to $\mathbb{R}$ and every half-open interval is homeomorphic to $[0,+\infty)$, and any continuous function on a closed bounded interval $[t_1,t_2]$ can be continuously extended to $\mathbb{R}$ by assigning the value $A(t_1)$ to each $t\leq t_1$ and the value $A(t_2)$ to each $t\geq t_2$.

We start by finding a strictly increasing sequence $0 =t_0<t_1<t_2<\cdots$ with $\lim_mt_m = \infty$ so that for all $m\in\N$ there exists $F_m\subset\{1,\ldots,N\}$ with
\begin{itemize}
\item[(a)] $\#F_m = n$ and
\item[(b)] for all $i\neq j\in F_m$ and $t_{m-1}\leq t\leq t_m$ we have $|\langle a_i(t),a_j(t)\rangle| < \varepsilon$.
\end{itemize}
This is achieved as follows. For each $r\in[0,1]$ we use Corollary \ref{combinatorial estimate zero} to find $F_r\subset\{1,\ldots,N\}$  so that for all $i\neq j\in F_r$ we have $|\langle a_i(t),a_j(t)\rangle| <\varepsilon$. Because $A$ is continuous, we may find a small open interval $I_r$ containing $r$ (half open if $r=0$) so that for all $i\neq j\in F_r$ and $t\in I_r$ we still have  $|\langle a_i(t),a_j(t)\rangle| <\varepsilon$. because $[0,1]\subset \cup_{r\in[0,1]}I_r$ and the interval $[0,1]$ is compact there must exist $r_1<\cdots<r_{m_1}$ so that $[0,1]\subset \cup_{i=1}^{m_1}I_{r_i}$. By perhaps getting rid of a few intervals we may assume that none of them is contained in the union of the others. Then, by perhaps making some of the intervals a little shorter we may assume that $\sup (I_{r_i})\leq r_{i+1}$ for $1\leq i <m_{1}-1$ and $ r_{i-1} \leq \inf(I_{r_i})$ for $1<i\leq m_1$. In other words, for $i= 1,\ldots,m_1-1$ we have $\emptyset \neq I_{r_i}\cap I_{r_{i+1}}\subset (r_i,r_{i+1})$. Define $t_0 = 0$, $t_{m_1} = 1$ and for $1\leq i<m_1$ pick $t_i\in (r_i,r_{i+1})$. If we then set $F_i = F_{r_i}$ for $1\leq i\leq m_1$ we obtain that (a) and (b) are satisfied up to $m=m_1$. For $k=2,3,\ldots$ repeat the same argument on $[k-1,k]$ to find $(t_i)_{i=m_{k-1}+1}^{m_k}$ and $(F_i)_{i=m_{k-1}+1}^{n_k}$ that satisfy (a) and (b).

The next step is to apply for each $m=1,2,\ldots$  Corollary \ref{combinatorial estimate one} to the matrix $A(t_m)$ and the sets $F_m$, $F_{m+1}$. By doing so we find a set $G_m\subset\{1,\ldots,N\}\setminus(F_m\cup F_{m+1})$ with $\#G_m = n$ so that for all $i\neq j$ with $i\in G_m$ and $j\in G_m\cup F_m\cup F_{m+1}$ we have $|\langle a_i(t_m),a_j(t_m)\rangle| < \varepsilon$. We now use the continuity of $A$ once more to find $s_m<t_m<u_m$ so that for all $t\in (s_m,u_m)$ the above hold as well.
By perhaps moving $s_m, u_m$ a bit closer to $t_m$ we have the following situation:
\begin{itemize}
\item[(c)] $0=t_0<s_1<t_1<u_1<s_2<t_2<u_2<s_3<t_3<u_3<\cdots$,
\item[(d)] for $m=1,2,\ldots$ we have $G_m\subset\{1,\ldots,N\}\setminus(F_m\cup F_{m+1})$ with $\#G_m = n$ so that for all $t\in(s_m,u_m)$, $i\neq j$ with $i\in G_m$ and $j\in G_m\cup F_m\cup F_{m-1}$ we have $|\langle a_i(t),a_j(t)\rangle| <\varepsilon$.
\end{itemize}

We are finally ready to define $L(t)$ and $R(t)$. Set $0 = u_0$. For each $m=1,2,\ldots$ take a continuous $\lambda_m:[s_m,u_m]\to[0,1]$ with $\lambda_m(s_m) = \lambda_m(u_m) = 1$ and $\lambda_m(t_m) = 0$.
\begin{itemize}
\item[(A)] For $m=0,1,\ldots $ and $t\in[u_m,s_{m+1}]$ set $R(t) = R_{(A(t),F_{m+1})}$.
\item[(B)] For $m=1,2,\ldots$ and $t\in[s_m,t_m]$ define $R(t) = R^{\lambda_m(t)}_{A(t),F_m,G_m}$. We point out that, by Remark \ref{endpoints are right}, $R(s_m) = R^{1}_{A(s_m),F_m,G_m} = R_{(A(s_m),F_m)}$ and $R(t_m) = R^0_{A(t_m),F_m,G_m} = R_{A(t_m),G_m}$.
\item[(C)] For $m=1,2,\ldots$ and $t\in[t_m,u_m]$ define $R(t) = R^{\lambda_m(t)}_{A(t),F_{m+1},G_m}$. Once more, by Remark \ref{endpoints are right}, $R(t_m) = R^{0}_{A(t_m),F_{m+1},G_m} = R_{(A(t_m),G_m)}$ and $R(u_m) = R^1_{A(u_m),F_{m+1},G_m} = R_{A(u_m),F_{m+1}}$.
\end{itemize}
By Remark \ref{still continuous}, in each case (A), (B), and (C) the function $R$ is continuous and the values at the endpoints of the corresponding intervals match. Thus $R$ defines a continuous function on $I$ and thus so does $L = (AR)^T$.

We next wish to show that for $t\geq 0$ we have $\|L(t)A(t)R(t) - I_n\| \leq 1/4$ and $\|L(t)\|\|R(t)\| \leq 4/(3\theta)$ and the proof will be complete. If $t\in [u_m,s_{m+1}]$, for some $m\in\mathbb{N}$, then this follows from Definition (A) above and \eqref{errors good zero} applied to Proposition \ref{F matrices estimates}. If $t\in[s_m,u_m]$ for some $m\in\N$ then this follows from Definition (B) or (C), property (d), and \eqref{errors good zero} applied to Proposition \ref{mix matrices estimates}.
\end{proof}

We conclude with some open questions regarding the topic of the paper.

\begin{question}
\label{probabilistic method}
As it was pointed out in Remark \ref{Bourain-Tzafriri remark}, \cite{Bourgain-Tzafriri-1987} implies a version of Theorem \ref{theorem stationary} (in which $2/\theta$ is replaced by $C/\theta$ and $C$ is a non-explicit finite constant) with an estimate $n \gtrsim \theta^2N$. This is better than our estimate $n \gtrsim \theta^{3/4}N^{1/3}$, provided that $N \gtrsim 1/\theta$.  Can the probabilistic technique from \cite{Bourgain-Tzafriri-1987} be used to obtain a similar version of the continuous Theorem \ref{main theorem} with an estimate $n \gtrsim \theta^2N$?
\end{question}

\begin{question}
For the theorem in the continuous case, we considered $A$ : $I \rightarrow M_N(\mathbb{R})$ where $I$ is an interval of $\mathbb{R}$. We conjecture that a version of Theorem \ref{main theorem} is also true for a continuous matrix function $A$ : $\mathbb{R}^d \rightarrow M_N(\mathbb{R})$. What is the relation between $d$, $N$, $\theta$, and the dimension $n$ in the conclusion of such a theorem?
\end{question}

For $1\leq p\leq\infty$ and an $N\times N$ matrix $A$ let $\|A\|_p$ denote the quantity $\max\{\|Ax\|_p:\|x\|_p\leq 1\}$. In particular, $\|A\| = \|A\|_2$.

\begin{question}
The methods used in this paper rely heavily on properties of the euclidean norm. In the statement of Theorem \ref{main theorem} we may replace condition (i) with $\|A(t)\|_p\leq 1$. It would be interesting to prove a version of this theorem as different methods might be necessary.
\end{question}

\subsection*{Acknowledgement} We would like to thank the anonymous referee for recommending the inclusion of Question \ref{probabilistic method}.

\bibliographystyle{abbrv}
\bibliography{bibliography}

\end{document}